\documentclass[11pt]{article}
\usepackage[T1]{fontenc}
\usepackage[utf8]{inputenc}
\usepackage{parskip}
\usepackage{lmodern}
\usepackage{amsmath,amssymb,amsthm}
\usepackage[margin=1in]{geometry}
\usepackage{microtype}
\usepackage{enumitem}
\usepackage{needspace}
\usepackage[hidelinks,bookmarksnumbered=true]{hyperref}
\newtheorem{theorem}{Theorem}
\newtheorem{lemma}[theorem]{Lemma}
\newtheorem{corollary}[theorem]{Corollary}
\theoremstyle{remark}
\newtheorem{remark}[theorem]{Remark}
\numberwithin{equation}{section}
\newcommand{\E}{\mathbb E}
\title{General-position sets and linear hypergraphs}
\author{Thang Pham \thanks{Institute of Mathematics and Interdisciplinary Sciences, Xidian University. \newline
\hspace*{0.45cm} Email: {\tt thangphammath@xidian.edu.cn}} \and Le Quang-Ham\thanks{The Viet Nam National Institute of Educational Sciences.  Email: {\tt hamlq@vnies.edu.vn}} \and Boqing Xue\thanks{Institute of Mathematical Sciences, ShanghaiTech University. ~Email: {\tt xuebq@shanghaitech.edu.cn}}}
\date{}
\hypersetup{pdftitle={General-position sets and linear hypergraphs},
  pdfauthor={Thang Pham, Le Quang-Ham, Boqing Xue}}
\begin{document}
\maketitle

\begin{abstract}
We prove that every linear hypergraph on $N$ vertices with edge sizes $O(\sqrt N)$ and $O(N^2\log N)$ triples contained in its edges has a subset of size $\Omega(\sqrt N)$ meeting each edge in at most two vertices. As an application, every set of $N$ planar points with at most $\sqrt N$ collinear points contains $\Omega(\sqrt N)$ points in general position. This proves a conjecture of Payne and Wood and gives the optimal quadratic bound for a geometric Ramsey problem posed by Gowers.
\end{abstract}

{\footnotesize
}

\section{Introduction}

A hypergraph $\mathcal H=(V,\mathcal E)$ is called \emph{linear} if distinct edges meet in at most one vertex. Its associated $3$-uniform hypergraph $H$ has vertex set $V$ and edge set
\[
 E(H)=\bigcup_{B\in\mathcal E}\binom B3.
\]
We call these edges \emph{forbidden triples}. A set $I\subseteq V$ is \emph{independent} in $H$ if it contains no forbidden triple, i.e.,  $|I\cap B|\le2$ for every $B\in\mathcal E$. Unless specified otherwise, edges refer to $\mathcal H$ and independence refers to $H$. All logarithms are natural.

Our fundamental result gives an independent set of square-root size under bounds on the edge sizes and the number of forbidden triples.

\begin{theorem}\label{thm:main}
For every $C_0\ge1$, there are constants $c>0$ and $N_0$ such that the following holds. Let $\mathcal H=(V,\mathcal E)$ be a linear hypergraph with $N=|V|\ge N_0$. If
\[
 \max_{B\in\mathcal E}|B|\le C_0\sqrt N,
 \qquad
 \sum_{B\in\mathcal E}\binom{|B|}{3}\le C_0N^2\log N,
\]
then there exists an independent set $I\subseteq V$ with $|I|\ge c\sqrt N$.
\end{theorem}

The order $\sqrt N$ is best possible. For $N=k^2$, take $k$ disjoint edges of size $k$. These edges contain $k\binom k3=O(N^2)$ forbidden triples, but every independent set has at most $2k=2\sqrt N$ vertices.

Subsets meeting each edge of a linear hypergraph in at most two vertices were studied by Axenovich, Brada\v c, Gishboliner, Mubayi and Weber~\cite[Section~5]{ABGMW}. When the edges have equal size, the associated hypergraph belongs to the class of unions of nearly disjoint hypergraph cliques studied by Mubayi and Verstra\"ete~\cite{MV}. Although $\mathcal H$ is linear, $H$ need not be: a pair in an edge $B$ belongs to $|B|-2$ forbidden triples. Thus the independence theorem of Duke, Lefmann and R\"odl~\cite{DLR} for uniform linear hypergraphs does not apply directly.

Under the hypotheses of Theorem~\ref{thm:main}, deleting at most half the vertices leaves a subhypergraph of $H$ with maximum degree $O(N\log N)$ and maximum pair codegree $O(\sqrt N)$. The coloring theorem of Cooper and Mubayi~\cite[Theorem~5]{CM} then gives an independent set of size
\[
 \Omega\!\left(\sqrt{\frac{N\log\log N}{\log N}}\right).
\]
Theorem~\ref{thm:main} removes this logarithmic loss by using the structure of the complete triple systems within the edges of $\mathcal H$.

We now give geometric applications of Theorem~\ref{thm:main}. A point set in $\mathbb F^2$, where $\mathbb F$ is a field, is said to be \emph{in general position} if it contains no three collinear points. We first consider point sets in the real plane $\mathbb R^2$. The following theorem resolves the conjecture of Payne and Wood~\cite[Conjecture~4.1]{PW}.

\begin{theorem}\label{thm:planar}
For every $A\ge1$, there are constants $c>0$ and $N_0$, depending only on $A$, such that every set of $N\ge N_0$ points in $\mathbb R^2$ with at most $A\sqrt N$ collinear points contains at least $c\sqrt N$ points in general position.
\end{theorem}

Payne and Wood~\cite{PW} proved that $N$ planar points with at most $s$ collinear points contain a general-position subset of size $\Omega(\sqrt{N/\log s})$ when $3\le s=O(\sqrt N)$. They also proved the stronger bound $\Omega(\sqrt{N\log N/\log s})$ for $s\le N^{(1-\varepsilon)/2}$, with any fixed $0<\varepsilon<1$. Thus the bound $\Omega(\sqrt N)$ was known when every line contains at most $N^\alpha$ points for some fixed $\alpha<1/2$. At $s=\sqrt N$, their bound gives only $\Omega(\sqrt{N/\log N})$, and they conjectured that the logarithmic loss could be removed. Balogh, Clemen, Dumitrescu and Liu~\cite[Proposition~2]{BCDL} subsequently improved this bound to $\Omega(\sqrt{N\log\log N/\log N})$, using the coloring theorem of Cooper and Mubayi~\cite{CM}. Theorem~\ref{thm:planar} removes the remaining logarithmic loss. The upper bound of Balogh, Clemen, Dumitrescu and Liu~\cite[Theorem~1]{BCDL} shows that the order $\sqrt N$ is optimal when the permitted collinearity is a fixed positive multiple of $\sqrt N$.

In 2011, Gowers~\cite{Gowers} asked how many planar points are needed to guarantee either $k$ collinear points or $k$ points in general position. The following theorem gives the optimal quadratic bound.

\begin{theorem}\label{thm:planar2}
There is an absolute constant $K_0>0$ such that, for every integer $k\ge3$, every set of at least $K_0k^2$ points in $\mathbb R^2$ contains either $k$ collinear points or $k$ points in general position.
\end{theorem}

A greedy argument gives an upper bound of order $k^3$. Payne and Wood~\cite{PW} improved this to $O(k^2\log k)$, and Hajnal and Szemer\'edi~\cite[Theorem~3]{HS} obtained $O(k^2\log k/\log\log k)$ by combining sparsification with independence estimates from the semi-random method. For the matching lower bound, put $m=\lfloor(k-1)/2\rfloor$. The grid $\{1,\ldots,m\}^2$ has $\Theta(k^2)$ points and at most $m<k$ points on any line. Every general-position subset contains at most two points in each row, hence fewer than $k$ points.

\begin{remark}\label{rem:characteristic-zero}
Theorems~\ref{thm:planar} and~\ref{thm:planar2} hold over every field of characteristic zero, with constants independent of the field. This follows from T\'oth's complex incidence theorem~\cite{Toth} (see Section~\ref{sec:applications}).
\end{remark}

Lewko's incidence theorem~\cite{Lewko} also gives an extension to positive characteristic.

\begin{corollary}\label{cor:finite-field}
For every $A\ge1$, there are constants $c>0$ and $N_0$, depending only on $A$, with the following property. Let $\mathbb F$ be a field of characteristic $p>0$, and let $V\subseteq\mathbb F^2$ have $N\ge N_0$ points. If $\frac{N}{\log N}\le p$ and $|V\cap\ell|\le A\sqrt N$ for every affine line $\ell$, then $V$ contains at least $c\sqrt N$ points in general position.
\end{corollary}

We outline the proof of Theorem~\ref{thm:main}. The bound on the number of forbidden triples allows us to delete at most half the vertices so that the remaining $3$-uniform hypergraph has maximum degree $O(N\log N)$. We then construct a random independent set in which each remaining vertex is selected with probability $\Omega(N^{-1/2})$. Summing these probabilities gives an expected size of at least a constant multiple of $\sqrt N$, so some independent set has the required size.

The construction has two stages. First, we sample vertices independently and remove those with unusually large sampled degree or belonging to an unusually large sampled edge. This reduces both degrees and edge sizes while retaining each vertex with probability comparable to the sampling probability. We then process the retained vertices in a fixed order. Each unprocessed vertex has a changing weight, which gives its selection probability when its turn arrives. Selecting a vertex sets to zero the weights of later vertices that would complete a forbidden triple with it and an already selected vertex. When the current vertex is not selected, we increase those same weights to compensate. All weights are kept at or below a small cap. These updates ensure independence and preserve each affected weight in conditional expectation whenever the cap does not apply.

The main difficulty is to control the loss caused by the cap. For each vertex, we introduce a nonnegative potential that depends on its weight and the weights on its incident edges. Its initial expectation is of order $N^{-1/2}$, and its final value is at most the indicator that the vertex is selected. Our one-step estimates show that any decrease in its conditional expectation is controlled by interactions between distinct edges. Using linearity and the degree bound, we bound the number of configurations that can give rise to these interactions. The initial random sampling makes such configurations unlikely to survive, while the cap bounds their contributions to the expected loss. Together, these estimates show that the total expected loss is negligible
compared with the initial expected potential. Summing the one-step estimates therefore gives the desired inclusion probability, as established in Lemma~\ref{lem:marginal}.

Our construction is a capped variant of the ordered selection process of Martinsson and Steiner~\cite[Section~4]{MS} and its hypergraph extension by Yu and Zhang~\cite[Algorithm~1]{YZ}. The logarithmic term in our potential is related to the entropy estimates of Dhawan~\cite{Dhawan}. We prove all the estimates needed here. Section~\ref{sec:proof} gives the proof of Theorem~\ref{thm:main}, and Section~\ref{sec:applications} proves the geometric applications.

\section{Proof of the independent-set theorem} \label{sec:proof}

\subsection{Combinatorial estimates}

Throughout this section, edges of $\mathcal H$ of size less than three are discarded, since they contribute no forbidden triples. For $A\subseteq V$, we restrict $\mathcal H$ to $A$ by replacing each edge $B\in\mathcal E$ with $B\cap A$ and discarding intersections of size less than three. Its associated $3$-uniform hypergraph is exactly $H[A]$, the subhypergraph of $H$ induced by $A$.

We use the convention $\binom{k}{r}=0$ for integers $k,r$ with $0\le k<r$. For $A\subseteq V$ and $v\in A$, let
\[
 d_A(v) := \deg_{H[A]}(v)=\sum_{\substack{B\in\mathcal E\\B \ni v}}\binom{|B\cap A|-1}{2}.
\]
Write $d(v)=d_V(v)$, and let $D\ge1$ be an upper bound for the degrees $d(v)$.

For distinct $u,v\in V$, let $B_{uv}$ be their common edge, if one exists, and let $B_{uv}=\varnothing$ otherwise. Linearity makes this edge unique when it exists. Set
\[
 w(u,v)=
 \begin{cases}
 |B_{uv}|-2,&B_{uv}\ne\varnothing,\\
 0,&B_{uv}=\varnothing,
 \end{cases}
 \qquad w(v,v)=0.
\]
Thus $w(u,v)$ counts the forbidden triples containing $u,v$. For every vertex $v$, we deduce that
\begin{equation}\label{eq:weight-identities}
 \sum_xw(v,x)=2d(v)\le2D,
 \qquad
 \sum_{B\ni v}(|B|-2)^2\le2d(v)\le2D.
\end{equation}
Indeed, an edge of size $b$ contributes $(b-1)(b-2)$ to the first sum, and $(b-2)^2\le(b-1)(b-2)$ to the second.

\begin{lemma}\label{lem:first-pencil}
If $B\in\mathcal E$ and $v\notin B$, then
\[
 \sum_{x\in B}w(v,x)\le\sqrt{2|B|D}.
\]
\end{lemma}

\begin{proof}
The nonempty edges $B_{vx}$, for $x\in B$, are distinct. Otherwise one such edge would meet $B$ in two points and would have to equal $B$, contrary to $v\notin B$. Therefore,
\[
 \sum_{x\in B}w(v,x)^2  \le\sum_{B'\ni v}(|B'|-2)^2\le2D.
\]
Cauchy--Schwarz now gives
\[
\Big(\sum_{x\in B}w(v,x)\Big)^2  \le |B|\sum_{x\in B}w(v,x)^2  \le 2|B|D.
\]
\end{proof}

\begin{lemma}\label{lem:pencil}
For distinct $u,v\in V$,
\[
 \sum_{x\in V\setminus B_{uv}}w(u,x)w(v,x)\le7D^{4/3}.
\]
\end{lemma}

\begin{proof}
For each edge $B\in\mathcal E$ containing $u$ but not $v$, let $\alpha_B=|B|-2$. By linearity and \eqref{eq:weight-identities},
\[
 \sum_{x\in V\setminus B_{uv}}w(u,x)w(v,x)=\sum_{\substack{B\ni u\\B\ne B_{uv}}}\alpha_B\sum_{x\in B\setminus\{u\}}w(v,x),
 \qquad
 \sum_{\substack{B\ni u\\B\ne B_{uv}}}\alpha_B^2\le2D.
\]
Split the edges according to whether $\alpha_B<D^{1/3}$. The edges with $\alpha_B<D^{1/3}$ contribute at most
\[
 D^{1/3}\sum_{x\in V}w(v,x)\le2D^{4/3}.
\]
For the other edges, Lemma~\ref{lem:first-pencil} and $|B|=\alpha_B+2\le3\alpha_B$ give the upper bound
\[
 \sum_{\alpha_B\ge D^{1/3}}\alpha_B\sqrt{2|B|D}  \le\sqrt{6D}\sum_{\alpha_B\ge D^{1/3}}\alpha_B^{3/2}  \le2\sqrt6\,D^{4/3}.
\]
Adding these contributions gives
\[
 \sum_{x\in V\setminus B_{uv}}w(u,x)w(v,x)\le(2+2\sqrt6)D^{4/3}<7D^{4/3}.
\]
\end{proof}

\begin{lemma}\label{lem:rooted-triangles}
Fix $v\in V$. Consider ordered tuples $(u_1,u_2,x,y_1,y_2)$ such that the six vertices are distinct and one of the following holds for edges $B,B_1,B_2$, with $x\notin B$:
\begin{align}
 &\{v,u_1,u_2\}\subseteq B,\qquad
   \{x,u_1,y_1\}\subseteq B_1,\qquad
   \{x,u_2,y_2\}\subseteq B_2;
   \label{eq:rooted-pattern-a}\\
 &\{v,u_1,u_2\}\subseteq B,\qquad
   \{x,v,y_1\}\subseteq B_1,\qquad
   \{x,u_1,y_2\}\subseteq B_2.
   \label{eq:rooted-pattern-b}
\end{align}
There are at most $14D^{7/3}$ configurations of each type.
\end{lemma}

\begin{proof}
There are
\[
 \sum_{B\ni v}(|B|-1)(|B|-2)=2d(v)\le2D
\]
choices of an edge $B$ and an ordered pair of distinct points $u_1,u_2\in B\setminus\{v\}$. For a fixed such choice and a fixed $x\notin B$, a configuration of type~\eqref{eq:rooted-pattern-a} has $B_1=B_{u_1x}$ and $B_2=B_{u_2x}$. The number of completions of \eqref{eq:rooted-pattern-a} is at most
\[
 \sum_{x\notin B}w(u_1,x)w(u_2,x)\le7D^{4/3}
\]
by Lemma~\ref{lem:pencil} with $B=B_{u_1u_2}$. For \eqref{eq:rooted-pattern-b}, it is instead
\[
 \sum_{x\notin B}w(v,x)w(u_1,x)\le7D^{4/3},
\]
noting that $B=B_{vu_1}$.

These counts do not introduce coincidences among the six points. In the first type, $y_1$ cannot equal $v$ or $u_2$: either equality would force $B_1=B$ by a two-point intersection, contradicting $x\notin B$. The same argument excludes $y_2=v$ and $y_2=u_1$. If $y_1=y_2$, then $B_1=B_2$, and this common edge contains $u_1,u_2$, again equaling $B$. All other coincidences are excluded by the choices of $u_1,u_2,x,y_1,y_2$. In the second type, the same argument excludes $y_1=u_1,u_2$, excludes $y_2=v,u_2$, and excludes $y_1=y_2$. It also shows that the three supporting edges are distinct. Each supporting edge is uniquely determined by two of its points, so no additional choices of edges need to be counted. Multiplying each bound by $2D$ proves the estimates.
\end{proof}

\subsection{The selection process}

We construct a random independent set by considering the vertices one at a time. Each unprocessed vertex has a weight, which determines its selection probability when its turn arrives. Selecting a vertex may prevent some later vertices from being selected. When the current vertex is rejected, we increase the weights of those later vertices to compensate for this possible loss. A cap keeps all selection probabilities small.

Let $Q\subseteq V$ be a random subset, whose distribution will be specified in the proof of Lemma~\ref{lem:marginal}. We describe the selection process conditional on $Q$. Fix an order on $V$ before sampling $Q$, and list the vertices of $Q$ in this order as $x_1,\ldots,x_{|Q|}$. Let $\xi_1,\ldots,\xi_{|V|}$ be independent uniform random variables
on $[0,1]$, independent of the sampling used to obtain $Q$. Only the first $|Q|$ of these variables are used.

Choose parameters $0<a<\tau\le1/4$. Initially, no vertex is selected, so $J_0=\varnothing$, and we assign every vertex $u\in Q$ the weight $z_u=a$. Suppose that $x_1,\ldots,x_j$ $(j=0,1,\ldots,|Q|-1)$ have been processed, and let $J_j\subseteq\{x_1,\ldots,x_j\}$ be the selected set. At step $j+1$, write $x=x_{j+1}$ and $q=z_x$. A later vertex $u\in\{x_{j+2},\ldots,x_{|Q|}\}$ is called \emph{affected} if there are an already selected vertex $y\in J_j$ and an edge $B\in\mathcal E$ such that $\{x,y,u\}\subseteq B$. Thus, if $x$ is selected, selecting $u$ later would complete a forbidden triple. The affected vertices are determined before $\xi_{j+1}$ is revealed.

We now decide whether to select $x$ and update the later weights.

\begin{enumerate}[label=(\roman*),leftmargin=*,itemsep=4pt]
\item
If $\xi_{j+1}<q$, select $x$ and set $J_{j+1}=J_j\cup\{x\}$. Set the weight of every affected vertex to zero.

\item
If $\xi_{j+1}\ge q$, reject $x$ and set $J_{j+1}=J_j$. For every affected vertex $u$, replace its current weight $z_u$ by the new weight $\min\left\{\frac{z_u}{1-q},\tau\right\}$.
\end{enumerate}

All other unprocessed weights remain unchanged. Each affected vertex is updated only once at a step, even if several choices of $y$ or $B$ show that it is affected. Processed vertices are never reconsidered. After all vertices have been processed, the output is $J=J_{|Q|}$.

Every unprocessed weight lies in $\{0\}\cup[a,\tau]$, and a weight that becomes zero remains zero. In particular, $q\le\tau<1$, so the denominator $1-q$ is positive.
The output is independent: whenever the second vertex of an edge $B$ is selected, every later vertex of $B\cap Q$ is affected and receives weight zero. No third vertex of $B$
can therefore be selected. Hence $|J\cap B|\le2$ for each $B\in\mathcal E$.

We next explain why an affected weight is divided by $1-q$ when the current vertex is not selected, before applying the cap $\tau$. Let $\mathcal F_j$ denote the information given by $Q$ and $\xi_1,\ldots,\xi_j$. Conditional on $\mathcal F_j$, the current vertex is selected with probability $q$. For an affected later vertex $u$, let $z'_u$ be its weight after this step. Then
\[
 \E[z'_u\mid\mathcal F_j]
 =(1-q)\min\left\{\frac{z_u}{1-q},\tau\right\}
 \le z_u.
\]
If $z_u/(1-q)\le\tau$, equality holds, since
\[
 q\cdot0+(1-q)\frac{z_u}{1-q}=z_u.
\]
Thus, when the cap does not apply, increasing the weight when $x$ is not selected exactly compensates in expectation for setting it to zero when $x$ is selected. The cap can decrease this expectation, but it ensures that every selection probability is at most $\tau$.

The main task is to show that the cap does not reduce the selection probabilities too much. Lemma~\ref{lem:marginal} chooses $Q$, $a$, and $\tau$ and uses a potential argument to obtain a lower bound for $\mathbb P(v\in J)$. This probability averages over both the sampling of $Q$ and the selection process. Summing these lower bounds gives the expected size needed for Theorem~\ref{thm:main}.

\Needspace{11\baselineskip}
\subsection{A lower bound for the selection probabilities}

\begin{lemma}\label{lem:marginal}
For every $K\ge1$, there is $n_0=n_0(K)$ such that the following holds for $n\ge n_0$. Let $\mathcal H=(V,\mathcal E)$ be a finite linear hypergraph satisfying
\[
 3\le |B|\le Kn\quad(B\in\mathcal E),
 \qquad d(v)\le Kn^2\log n\quad(v\in V).
\]
There is a random independent set $J\subseteq V$ such that
\[
 \mathbb P(v \in J)\ge\frac{\gamma}{4n}\quad(v\in V),
 \qquad \gamma=\frac1{1000K}.
\]
There is no restriction on $|V|$.
\end{lemma}

\begin{proof}
Set $D=Kn^2\log n$ and
\[
 p=n^{-5/6},\qquad a=\gamma n^{-1/6},\qquad
 \tau=n^{-19/120}.
\]
All estimates below hold for sufficiently large $n$ depending only on $K$; in particular, $0<a<\tau\le1/4$.

\smallskip\noindent\emph{Sparsification.}
Sample each vertex independently with probability $p$, obtaining a set $S$. Delete $v\in S$ if $d_S(v)>20p^2D$ or if $v\in B$ for some edge $B\in \mathcal E$ with $|B\cap S|>3Kpn$. Let $Q$ be the remaining set. Run the selection process on $Q$, using uniform random variables independent of
$S$, and let $J$ be its output.

Conditional on $v\in S$, each triple containing $v$ has its other two vertices in $S$ with probability $p^2$. Hence
\[
 \E[d_S(v)\mid v\in S]=p^2d(v)\le p^2D.
\]
Thus Markov's inequality bounds the probability of deletion for large degree by $1/20$. At most $D$ edges contain $v$. For each such edge $B$, the variable $|B\cap S|-1$, conditional on $v\in S$, is binomial with mean at most $Kpn$. The Chernoff bound~\cite[Theorem~2.4.7(i)]{Roch} gives
\[
 \mathbb P(|B\cap S|>3Kpn\mid v\in S)\le e^{-Kpn/3}.
\]
A union bound therefore yields
\begin{equation}\label{eq:general-retention}
 \mathbb P(v\in Q)\ge p\bigl(1-1/20-De^{-Kpn/3}\bigr)\ge p/2
\end{equation}
for sufficiently large $n$.

Let $\mathcal E_Q$ consist of the intersections $B\cap Q$ of size at least three. The linear hypergraph $(Q,\mathcal E_Q)$ has maximum edge size at most $s'=3Kpn$, and $H[Q]$ has maximum degree at most $D'=20p^2D$. Moreover, any six distinct vertices all lie in $Q$ with probability at most $p^6$, since $Q\subseteq S$.

\smallskip\noindent\emph{The potential.}
Set
\[
 \eta=\frac18,\qquad \sigma=\frac1{128\log n},\qquad
 \lambda=e^{-\eta/(2\sigma)}=n^{-8},\qquad
 \Lambda=\log(\tau/a)=\frac{\log n}{120}-\log\gamma.
\]
We have $\tau\le\sigma$, $2\tau\eta/\sigma\le\log2$, and $\Lambda\ge1$. Fix $v\in V$. We define a nonnegative potential $C_v$ at each step. If $v\notin Q$, set $C_v=0$ throughout the process. For each realization of $Q$ containing $v$, consider
the edges in $\mathcal E_Q$ containing $v$. Before $v$ is processed, let $U_B$ be the unprocessed vertices of
$B\setminus\{v\}$, and let $b_B$ count the already selected vertices in $B$. Define
\[
 G_B=\prod_{u\in U_B}(1+z_u/\sigma),\qquad
 \Phi_{B,0}=2\sum_{\{u_1,u_2\}\subseteq U_B}z_{u_1}z_{u_2},\qquad
 \Phi_{B,1}=2\sum_{u\in U_B}z_u.
\]
Whenever $z_v>0$, we have $b_B\in\{0,1\}$ for every such edge. Before $v$ is processed, as long as $z_v>0$ and $C_v$ has not previously become zero, define
\[
 C_v=z_v\left[1-\lambda\sum_{B\ni v}G_B -\frac{\log(z_v/a)+\sum_{B\ni v}\Phi_{B,b_B}}{\Lambda} \right]_+,
\]
where $[t]_+=\max\{t,0\}$ denotes the positive part of $t$. Set $C_v=0$ if $z_v=0$. Once $C_v$ becomes zero, keep it equal to zero for the remainder of the process. At the step when $v$ is processed, replace a positive potential by the selection indicator of $v$, and keep this value fixed thereafter. All sums over edges here and in the estimates below are taken over $\mathcal E_Q$, subject to the indicated conditions.

Before $v$ is processed, $0\le C_v\le z_v$. Thus the replacement at its own step does not decrease the conditional expectation, and the final potential is at most $\mathbf1_{\{v\in J\}}$. Moreover, $C_v>0$ implies $z_v<\tau$ and
\[
 M_B:=\sum_{u\in U_B}z_u<\eta\qquad(B\ni v).
\]
Indeed, $z_v=\tau$ makes the logarithm equal to $\Lambda$, while $M_B\ge\eta$ gives $\lambda G_B\ge1$, because $\log(1+z_u/\sigma)\ge z_u/(2\sigma)$. In particular,
$\lambda G_B\le1$ whenever $C_v>0$.

Initially, at most $D'$ edges in $\mathcal E_Q$ contain $v$, and
\[
 \sum_{B\ni v}\Phi_{B,0}=2a^2d_Q(v)\le2a^2D',\qquad
 as'\le3K\gamma\le\eta/4.
\]
Consequently,
\[
 \lambda\sum_{B\ni v}G_B
 \le D'e^{-\eta/(4\sigma)}=D'n^{-4}=o(1),
 \qquad
 \frac{2a^2D'}{\Lambda}
 \le4800K\gamma^2<\frac1{16}.
\]
For large $n$, the initial potential is therefore at least $3a/4$ when $v\in Q$. By \eqref{eq:general-retention}, its unconditional expectation is at least $3pa/8$.

\smallskip\noindent\emph{Changes in one step.}
Consider a step before $v$ is processed. Let $x$ be the current vertex, let $q=z_x$, and let $\mathcal F$ be the sigma-algebra generated by $Q$ and all random choices before this step. Primes denote the values after the step. We need only consider $C_v>0$. For a nonnegative vector $t$, let $e_j(t)$ denote its elementary symmetric polynomial of degree $j$, with $e_0(t)=1$. Write $t_A=(t_u)_{u\in A}$ for its restriction to an index set $A$.

First consider an edge $B$ containing $v$ but not $x$, and let $A\subseteq U_B$ be the vertices affected by $x$. Write $r_u=z_u/\sigma$. Omitting the cap gives an upper bound for the updated values of $G_B$ and $\Phi_{B,b_B}$. Expanding the product, we obtain
\[
 \E[G'_B\mid\mathcal F]-G_B  \le\prod_{u\in U_B\setminus A}(1+r_u) \sum_{j\ge2}e_j(r_A)\bigl((1-q)^{1-j}-1\bigr).
\]
For $j\ge2$,
\[
 (1-q)^{1-j}-1\le(j-1)q(1-q)^{1-j},\qquad
 \sum_{j\ge2}(j-1)e_j(t)  \le\sum_{j\ge2}\binom j2e_j(t)  \le e_2(t)\prod_u(1+t_u).
\]
The last inequality follows by retaining products with disjoint sets of indices. The ratio of the products with arguments $r_u/(1-q)$ and $r_u$ is at most $\exp(2qM_B/\sigma)\le2$. Hence
\[
 \E[G'_B\mid\mathcal F]-G_B  \le\frac{4qG_B}{\sigma^2}  \sum_{\{u_1,u_2\}\subseteq A}z_{u_1}z_{u_2}.
\]
Similarly,
\[
 \E[\Phi'_{B,0}\mid\mathcal F]-\Phi_{B,0}  \le\frac{2q}{1-q}e_2(z_A),\qquad
 \E[\Phi'_{B,1}\mid\mathcal F]\le\Phi_{B,1}.
\]

\emph{Case 1: $x$ does not affect $v$.} Then $z_v$ remains unchanged. By linearity, at most one edge $B_*$ contains both $v$ and $x$. If it exists, then $b_{B_*}=0$, and no later vertex of $B_*$ is affected by $x$. Writing $R=U_{B_*}\setminus\{x\}$, we have $G'_{B_*}=G_{B_*}/(1+q/\sigma)$ and
\[
 \E[\Phi'_{B_*,b'_{B_*}}\mid\mathcal F]  =2(1-q)e_2(z_R)+2q\sum_{u\in R}z_u  \le\Phi_{B_*,0}.
\]
For every other edge $B$ containing $v$, let $A_B\subseteq U_B$ be its vertices affected by $x$. Applying the preceding estimates and $[t]_+\ge t$ gives
\[
 \E[C'_v\mid\mathcal F]\ge C_v-  z_v\sum_{\substack{B\ni v\\x\notin B}}  \left(\frac{4\lambda qG_B}{\sigma^2}  +\frac{2q}{(1-q)\Lambda}\right)e_2(z_{A_B}).
\]
The quadratic bound also applies when $b_B=1$, since the conditional expected change of the linear term is nonpositive. Using $\lambda G_B\le1$, $\Lambda\ge1$, $\sigma\le1$, and
$q\le1/4$, we conclude that
\begin{equation}\label{eq:first-loss}
 \E[C'_v\mid\mathcal F]\ge C_v-  \frac{8\tau^3q}{\sigma^2}  \sum_{\substack{B\ni v\\x\notin B}}  \bigl|\{\{u_1,u_2\}\subseteq U_B: x\text{ affects both }u_1,u_2\}\bigr|.
\end{equation}

\emph{Case 2: $x$ affects $v$.} Its unique common edge $B_*$ with $v$ has $b_{B_*}=1$. Selection of $x$ makes $C'_v=0$. On rejection, every vertex in $R=U_{B_*}\setminus\{x\}$ is affected. With $M_R=\sum_{u\in R}z_u$, omitting the caps gives
\[
 \log(G'_{B_*}/G_{B_*})  \le-\log(1+q/\sigma)+\frac{qM_R}{\sigma(1-q)}\le0.
\]
Indeed, $\log(1+q/\sigma)\ge q/(2\sigma)$ and $M_R/(1-q)\le\eta/(1-\tau)\le1/6$. Also,
\[
 \Phi_{B_*,1}-\Phi'_{B_*,1}  \ge2(q+M_R)-\frac{2M_R}{1-q}  =\frac{2q(1-M_{B_*})}{1-q}  \ge\frac{q}{1-q}\ge-\log(1-q).
\]
Since
\[
 \sum_{B\ni v}\Phi_{B,b_B}\ge\Phi_{B_*,1}  \ge2q\ge-\log(1-q),
\]
positivity of $C_v$ implies $\log(z_v/a)-\log(1-q)<\Lambda$. Thus the cap does not apply to the new weight of $v$, and
\[
 z'_v=\frac{z_v}{1-q}<\tau,\qquad  \log(z'_v/a)+\Phi'_{B_*,1}  \le\log(z_v/a)+\Phi_{B_*,1}.
\]
For any other edge $B$ containing $v$, let $A_B\subseteq U_B$ be its affected vertices. On rejection,
\[
 G'_B-G_B\le\frac{4qG_B}{\sigma}\sum_{u\in A_B}z_u,
 \qquad
 \Phi'_{B,b_B}-\Phi_{B,b_B}\le10q\sum_{u\in A_B}z_u.
\]
For the product bound, the weights do not decrease when $x$ is not selected, and
\[
 0\le\log\frac{G'_B}{G_B}  \le\frac{2q}{\sigma}\sum_{u\in A_B}z_u  \le\log2.
\]
Applying $e^t-1\le2t$ for $0\le t\le\log2$, with $t=\log(G'_B/G_B)$, gives
\[
 G'_B-G_B  \le2G_B\log\frac{G'_B}{G_B}  \le\frac{4qG_B}{\sigma}\sum_{u\in A_B}z_u.
\]
We next estimate the changes in $\Phi_{B,1}$ and $\Phi_{B,0}$. The linear term increases by at most $2q(1-q)^{-1}\sum_{u\in A_B}z_u$. For the quadratic term, every changed pair meets $A_B$, and its multiplicative increase is at most $(1-q)^{-2}-1\le4q$. Its total increase is therefore at most $8qM_B\sum_{u\in A_B}z_u$. Combining these bounds with $(1-q)z'_v=z_v$ yields
\begin{equation}\label{eq:second-loss}
 \E[C'_v\mid\mathcal F]\ge C_v- \frac{16\tau^2q}{\sigma}  \sum_{\substack{B\ni v\\B\ne B_*}}  \bigl|\{u\in U_B:x\text{ affects }u\}\bigr|.
\end{equation}

\smallskip\noindent\emph{Summing the losses.}
We average over both the selection process and the sampling. Every restricted edge comes from a unique original edge, by linearity. In the following configuration count, $B,B_*,B_1,B_2$ denote the corresponding original edges. For an incidence counted in \eqref{eq:first-loss}, order $u_1,u_2$ by the fixed order, and choose the first already selected witness $y_i$ for the effect of $x$ on $u_i$, for $i=1,2$. Then
\[
 \{v,u_1,u_2\}\subseteq B,\qquad
 \{x,u_1,y_1\}\subseteq B_1,\qquad
 \{x,u_2,y_2\}\subseteq B_2,\qquad x\notin B.
\]
This is the first configuration in Lemma~\ref{lem:rooted-triangles}.

For an incidence in \eqref{eq:second-loss}, denote the affected vertex by $u_1$, and choose the first vertex $u_2\in(B\cap Q)\setminus\{v,u_1\}$. This set is nonempty because
the restricted edge has size at least three. Choose the first selected witnesses $y_1,y_2$ for the effects of $x$ on $v,u_1$, respectively. Then
\[
 \{v,u_1,u_2\}\subseteq B,\qquad
 \{x,v,y_1\}\subseteq B_*,\qquad
 \{x,u_1,y_2\}\subseteq B_2,\qquad x\notin B.
\]
This is the second configuration in that lemma. In both cases, all six vertices belong to $Q$ and are distinct: the three edges meet at three different vertices, and linearity prevents any further intersection. The auxiliary vertex $u_2$ in the second case may already have been processed, which does not affect this argument.

Each tuple determines the counted edge, the affected vertex or pair, and the current vertex $x$. It is therefore assigned at most once in any run. By Lemma~\ref{lem:rooted-triangles}, there are at most $14D^{7/3}$ possible tuples of each type in the original hypergraph.

Only tuples with both witnesses preceding $x$ can contribute. Fix one such tuple, and let $T=\{v,u_1,u_2,x,y_1,y_2\}$ be its set of six vertices. Fix a realization $Q=Q_0$ of positive probability with $T\subseteq Q_0$, and let $q_x$ be the weight of $x$ immediately before it is processed. Then
\[
 \E\!\left[
 q_x\mathbf1_{\{y_1,y_2\}\subseteq J}
 \,\middle|\,Q=Q_0
 \right]
 =
 \mathbb P(x,y_1,y_2\in J\mid Q=Q_0)
 \le\tau^3.
\]
The equality follows by conditioning on the history immediately before $x$ is processed: both witnesses have already been processed, and the conditional probability
of selecting $x$ is $q_x$. The inequality follows by conditioning successively at the three processing steps, each of which has conditional selection probability at most $\tau$.

For an upper bound, we may omit the condition $C_v>0$ and require only that $y_1,y_2$ were selected before $x$, without requiring them to be the first selected witnesses
in the fixed order. The tuple contributes nothing unless $T\subseteq Q$. Averaging the preceding bound over the realizations $Q_0$ containing $T$, its contribution to
the expected sum of the factors $q$ is therefore at most
\[
 \tau^3\mathbb P(T\subseteq Q)
 \le\tau^3\mathbb P(T\subseteq S)
 =p^6\tau^3.
\]
The last equality uses the independent sampling of the six distinct vertices of $T$.  Multiplying by $8\tau^3/\sigma^2$ or $16\tau^2/\sigma$ and summing over the tuples bounds the total expected loss by $256p^6D^{7/3}\left(\tau^6/\sigma^2+\tau^5/\sigma\right)$. After division by $pa$, the two terms are respectively
\[
 O_K\bigl(n^{-17/60}(\log n)^{13/3}\bigr),
 \qquad
 O_K\bigl(n^{-1/8}(\log n)^{10/3}\bigr).
\]
Both tend to zero. Extend each potential by its final value to step $|V|$. Take expectations in the inequalities for the change of $C_v$ and sum over all steps. Since the initial expectation is at least $3pa/8$ and the final potential is at most $\mathbf1_{\{v\in J\}}$, we obtain
\[
 \mathbb P(v\in J)\ge\frac{3pa}{8}-o(pa)  \ge\frac{pa}{4}=\frac{\gamma}{4n}
\]
for all sufficiently large $n$.
\end{proof}

\subsection{Completion of the proof}

\begin{proof}[Proof of Theorem~\ref{thm:main}]
Write $n=\sqrt N$. Since distinct edges of $\mathcal H$ cannot contain the same triple,
\[
 \sum_{v\in V}d(v)  =3\sum_{B\in\mathcal E}\binom{|B|}{3}  \le6C_0n^4\log n.
\]
Let $V_0=\{v\in V:d(v)\le12C_0n^2\log n\}$. Then $|V_0|\ge n^2/2$. Restrict $\mathcal H$ to $V_0$, discarding edges of size less than three. Its edges have size at most $C_0n$ and its associated hypergraph is $H[V_0]$. Lemma~\ref{lem:marginal}, with $K=12C_0$, gives a random independent set $J\subseteq V_0$ for which
\[
 \E|J|\ge\frac{\gamma|V_0|}{4n}\ge\frac{\gamma n}{8}.
\]
Some realization has at least $(\gamma/8)\sqrt N$ vertices. Taking $c=\gamma/8$ proves the theorem.
\end{proof}

\section{Proofs of the geometric applications}
\label{sec:applications}

We first collect the required collinear-triple estimates. Over $\mathbb R$, the following bound is due to Payne and Wood~\cite[Lemma~2.1]{PW}, using the Szemer\'edi--Trotter theorem~\cite{ST}. The other cases follow from the incidence theorems of T\'oth~\cite{Toth} and Lewko~\cite{Lewko}. An incidence bound of the same order over prime fields was obtained independently by Miao and Xie~\cite{MiaoXie}.

\begin{lemma}\label{lem:geometric-triples}
Let $V\subseteq\mathbb F^2$ be a finite set of $N\ge2$ points with at most $s\ge2$ points on any line. Write $p$ for the characteristic of $\mathbb F$, with $p=\infty$ in characteristic zero and $1/\infty=0$. Then
\[
 \sum_{\ell}\binom{|V\cap\ell|}{3}  =O\!\left(N^2\log s+Ns^2+\frac{N^3}{p}\right),
\]
where the sum is over lines containing at least three points of $V$, and the implied constant is absolute.
\end{lemma}

\begin{proof}
We may assume that $s$ is an integer. For $j\ge2$, let $\mathcal L_j$ be the lines containing at least $j$ points of $V$. Counting pairs gives
\begin{equation}\label{eq:pair-rich-lines}
 |\mathcal L_j|\le\frac{N(N-1)}{j(j-1)}.
\end{equation}

Suppose first that $\mathbb F$ has characteristic zero. The coordinates of $V$ generate a finitely generated extension of $\mathbb Q$, which embeds in $\mathbb C$. Such an embedding preserves collinearity, so T\'oth's incidence theorem gives
\[
 j|\mathcal L_j|  \le C\bigl((N|\mathcal L_j|)^{2/3}+N+|\mathcal L_j|\bigr)
\]
for an absolute constant $C$. For $j\ge2C$, absorb the last term into the left-hand side. For smaller $j$, use \eqref{eq:pair-rich-lines}. Thus $|\mathcal L_j|=O(N^2/j^3+N/j)$.
The identity
\begin{equation}\label{eq:triple-sum}
 \sum_{\ell}\binom{|V\cap\ell|}{3}  =\sum_{j=3}^{s}\binom{j-1}{2}|\mathcal L_j|
\end{equation}
now gives $O(N^2\log s+Ns^2)$.

Suppose next that $\mathbb F$ has positive characteristic $p$. Lewko's theorem gives
\[
 j|\mathcal L_j|  \le3(N|\mathcal L_j|)^{2/3}+N+|\mathcal L_j|  +\frac{2N|\mathcal L_j|}{p}.
\]
Set $j_0=\max\{3,\lceil2+4N/p\rceil\}$. For $j\ge j_0$, one has $j|\mathcal L_j|/2  \le 3(N|\mathcal L_j|)^{2/3}+N$, and hence $|\mathcal L_j|=O(N^2/j^3+N/j)$. Use \eqref{eq:pair-rich-lines} for $3\le j<j_0$ and this last bound for $j\ge j_0$ in \eqref{eq:triple-sum}. The result is
\[
 O\!\left(N^2j_0+N^2\log s+Ns^2\right)  =O\!\left(N^2\log s+Ns^2+\frac{N^3}{p}\right).
\]
\end{proof}

\begin{proof}[Proof of Theorem~\ref{thm:planar} and Corollary~\ref{cor:finite-field}]
Let $N=|V|$, and let $\mathcal E$ consist of the sets $V\cap\ell$ of size at least three. Distinct lines meet in at most one point, so $\mathcal H=(V,\mathcal E)$ is linear. Let $H$ be the associated $3$-uniform hypergraph. Its edges are exactly the collinear triples of $V$, so its independent sets are precisely the subsets of $V$ in general position.

Since every edge has size at most $A\sqrt N$, Lemma~\ref{lem:geometric-triples} gives
\[
 \sum_{B\in\mathcal E}\binom{|B|}{3}  =O_A(N^2\log N)
\]
in characteristic zero, and also in positive characteristic when $N\le p\log N$. Theorem~\ref{thm:main}, with $C_0$ chosen in terms of $A$, proves both assertions. The same argument over arbitrary characteristic-zero fields proves the extension of Theorem~\ref{thm:planar} in Remark~\ref{rem:characteristic-zero}.
\end{proof}

\begin{proof}[Proof of Theorem~\ref{thm:planar2}]
Apply Theorem~\ref{thm:planar} with $A=1$, and choose $K_0\ge\max\{1,N_0,c^{-2}\}$. If $N=|V|\ge K_0k^2$ and $V$ contains no $k$ collinear points, then every line contains at most $k-1\le\sqrt N$ points of $V$. Thus $V$ contains at least $c\sqrt N\ge k$ points in general position. The same argument works over every characteristic-zero field, completing the proof of Remark~\ref{rem:characteristic-zero}.
\end{proof}

\section*{Acknowledgements}
The first author thanks Claudiu Valculescu for many fruitful discussions, particularly during his time at EPFL.

\end{document}